\documentclass[12pt]{article}
\usepackage{amsmath}
\usepackage{stmaryrd}
\usepackage{mathrsfs}
\usepackage{amssymb}
\usepackage{titlesec}

\usepackage{amsthm}

\titleformat{\section}[hang]{\small\bf}{\thesection}{1em}{}

\newtheorem{definition}{Definition}[section]

\newtheorem{theorem}{Theorem}[section]
\newtheorem{lemma}{Lemma}[section]

\begin{document}

    \title{ Katok's Entropy Formula of Unstable Metric Entropy for Partially Hyperbolic Diffeomorphisms
 \footnotetext {* Corresponding author}}
\author{Ping Huang$^{1,2}$, Ercai Chen$^{*2,3}$, Chenwei Wang$^{1}$\\
\small 1 College of Mathematical and Physical Sciences, Taizhou University,\\
     \small  Taizhou 225300, Jiangsu, P.R. China.\\
  \small   2 School of Mathematical Sciences and Institute of Mathematics, Nanjing Normal University,\\
   \small   Nanjing 210023, Jiangsu, P.R.China.\\
    \small 3 Center of Nonlinear Science, Nanjing University,\\
     \small   Nanjing 210093, Jiangsu, P.R.China.\\
      \small    e-mail: pinghuang1984@163.com ecchen@njnu.edu.cn chenweiwang01@163.com
}
\date{}
\maketitle

    \noindent{\small \textbf{Abstract}:
 Katok's entropy formula is an important formula in entropy theory. This paper is devoted to establishing   Katok's entropy formula of unstable metric entropy which is the entropy caused by the unstable part of partially hyperbolic systems.

    \noindent \textbf{Keywords}:
 Katok's entropy formula, unstable metric entropy, partially hyperbolic systems,  measure decomposition}

\section{Introduction}
 Let a triple $(X,d,f)$   be a topological dynamical system  in the sense that $f: X\rightarrow X$ is a continuous map on the compact metric space $X$ with metric $d$.   For $x,y\in X$ and $n\in \mathbb N$, the Bowen metric $d_n$ is given by
 \begin{equation*}
  d_n(x,y)=\max\{d(f^i(x),f^i(y)):i=0,1,\cdots,n-1\}.
 \end{equation*}
  Given $ \varepsilon >0$, $x\in X$ and $n\in \mathbb N$, let $B_{n}(x,\varepsilon)=\{y\in X : d_n(x,y)<\varepsilon\}$ denote the $d_n$-ball centered at $x$ with radius $\varepsilon$.

  In classical ergodic theory, measure-theoretic entropy and topological entropy are very important determinants of complexity in dynamical systems. And the variational principle reveals the relationship between the measure-theoretic entropy and the topological entropy. Katok's entropy formula which plays an important role in the study of entropy theory is an equivalent definition of the measure-theoretic entropy in a manner analogous to the definition of the topological entropy.

In 1980, Katok \cite{[Katok]} introduced the Katok's entropy formula: for any  $f$-invariant ergodic Borel probability measure $\mu$, and $0< \delta <1$,
\begin{equation*}\label{12}
\lim\limits_{\varepsilon \rightarrow 0}\limsup\limits_{n \rightarrow \infty }\frac{\log N_{\mu}(n, \varepsilon, \delta)}{n}=\lim\limits_{\varepsilon \rightarrow 0}\liminf\limits_{n \rightarrow \infty }\frac{\log N_{\mu}(n, \varepsilon, \delta)}{n}=h_{\mu}(f),
\end{equation*}
where $N_{\mu}(n, \varepsilon, \delta)$ denotes the minimal number of $d_n$-balls with radius $\varepsilon$ whose union is  a set of $\mu$-measure more than or equal to $1-\delta$.

In 2004, using spanning sets, He, Lv and Zhou \cite{[He]} introduced a  definition of measure-theoretic pressure of additive potentials for ergodic  measures, and obtained a pressure version of  Katok's entropy formula. In 2009,  Zhao and Cao \cite{[yunzhao]} gave a definition of  measure-theoretic pressure of  sub-additive potentials for ergodic  measures, and generalized the above results in \cite{[Katok]} and \cite{[He]}. Moreover, we refer to \cite{[zhao], [Cao]} for more  pressure versions of Katok's entropy formula. In 2009, Zhu  \cite{[Zhu2]} established a random version of Katok's entropy formula. In 2017, Zheng, Chen and Yang \cite{[zheng]} introduced an amenable version of Katok's entropy formula for countable discrete amenable group actions.  Huang, Wang and Ye \cite{[Huang]}  established Katok's entropy formula for ergodic  measures in the case of mean metrics.  In 2018,  Huang, Chen and Wang \cite{[Ping1]} constructed Katok's entropy formula of conditional entropy in mean metrics, where the conditional entropy is with respect to a $T$-invariant sub-$\sigma$-algebra. In 2019, Huang and Wang \cite{[Ping2]}  established a pressure version of Katok's entropy formula in the case of mean metrics.

Let $M$ be an $n$-dimensional smooth, connected and compact Riemannian manifold without boundary and $f: M \rightarrow M$ be a $C^1$-diffeomorphism. $f$ is  said to be partially hyperbolic if there exists a nontrivial $Df$-invariant splitting $TM=E^s\oplus E^c\oplus E^u$ of the tangent bundle into stable, center, and unstable distributions, such that all unit vectors $v^{\sigma}\in E^{\sigma}_x$ ($\sigma=c, s, u$) with $x\in M$ satisfy

\begin{equation*}
  \Vert D_xfv^s\Vert<\Vert D_xfv^c\Vert <\Vert D_xfv^u\Vert,
\end{equation*}
and
\begin{equation*}
  \Vert D_xf|_{E_x^s}\Vert<1 \quad \textrm{and} \quad \Vert D_xf^{-1}|_{E_x^u}\Vert<1,
\end{equation*}
for some suitable Riemannian metric on $M$. The stable distribution $E^s$ and unstable distribution $E^u$ are integrable to the stable and unstable foliations $W^s$ and $W^u$ respectively such that $DW^s=E^s$ and $DW^u=E^u$.

In 1985, from the measure theoretic point of view, Ledrappier and Young \cite{[Young]} gave a definition of unstable metric entropy for partially hyperbolic diffeomorphisms. And the entopy defined in \cite{[Young]} can be regarded as that given by $H_{\mu}(\alpha|f\alpha)$, where $\alpha$ is an increasing partition (i.e. $\alpha\geq f\alpha$)  subordinate to the
unstable leaves.
In 2008, from the topological point of view,  Hua, Saghin and Xia \cite{[Hua]} provided the unstable volume growth.
In 2017, Hu, Hua and Wu \cite{[HU]} introduced the definition of unstable metric entropy $h^u_{\mu}(f)$ for any
invariant measure $\mu$, and gave the definition of the unstable topological entropy $h^u_{\rm{top}}(f)$. Similar to that in the classical entropy theory, the corresponding versions of Shannon-McMillan-Breiman theorem, and the variational principle relating $h^u_{\mu}(f)$ and $h^u_{\rm{top}}(f)$ are given.
The unstable metric entropy $h^u_{\mu}(f)$ for an invariant measure $\mu$ is defined by using $H_{\mu}(\bigvee_{i=0}^{n-1}f^{-i}\xi|\eta)$, where $\xi$ is a finite measurable partition of the underlying manifold $M$, and $\eta$ is a measurable partition consisting of local unstable leaves that can be obtained by refining a finite partition into pieces of unstable leaves. In \cite{[HU]}, they showed that the unstable metric entropy  $h^u_{\mu}(f)$ is identical to the metric entropy defined in \cite{[Young]}.
Hu, Wu, Zhu \cite{[HU2]} introduced the notion of unstable topological pressure
 for a $C^1$-partially hyperbolic diffeomorphism $f : M \rightarrow M$ with respect to a  continuous function  on $M$, obtained a variational principle for this pressure. And they investigatedthe corresponding so-called $u$-equilibriums.
Motivated by the work of Bowen \cite{[Bo1]} and Pesin \cite{[Pesin2]}, Tian and Wu \cite{Tian} established a concept called Bowen unstable topological entropy which is the unstable topological entropy for any subsets (not necessarily compact or invariant) in partially hyperbolic systems as a Carath$\acute e$odory dimension characteristic. In particular, they proved the Bowen unstable topological entropy of the whole space coincides with the unstable topological entropy of the system in \cite{[HU]}.
And they also constructed some basic results in dimension theory for Bowen unstable topological entropy including a variational principle for any compact (not necessarily invariant) subset between its Bowen unstable topological entropy and unstable metric entropy of probability measures supported on this set.
 Wang, Wu and Zhu \cite{[Wang]} introduced the unstable entropy and unstable pressure for a random dynamical system with $u$-domination. For random diffeomorphisms with domination,  they also gave a version of Shannon-McMillan-Brieman theorem for  unstable metric entropy, and obtained a variational principle for unstable  pressure. In 2019, Wu \cite{[Wu]} introduced two notions of local unstable metric entropies and the notion of local unstable topological entropy relative to a Borel cover $\mathcal U$ of $M$, and showed that when $\mathcal U$ is an open cover with small diameter, the entropies coincide with the unstable metric entropy and unstable topological entropy, respectively.  And the unstable tail entropy and the unstable topological conditional entropy were also defined in \cite{[Wu]}.

 In this paper, inspired by the ideas of Katok \cite{[Katok]},  for ergodic measures, we establish  the Katok's entropy formula of unstable metric entropy $h^u_{\mu}(f)$  for partially hyperbolic diffomorphisms.

  The following theorems present the main results of this paper.

\begin{theorem}\label{theorem2}$($Katok's entropy formula  of unstable metric entropy$)$

Let $M$ be an $n$-dimensional smooth, connected and compact Riemannian manifold without boundary and  $f : M \rightarrow M$ be a $C^1$-partially hyperbolic diffeomorphism.
 Suppose $\mu$ is an ergodic measure of $f$. Let $\eta\in \mathcal P^u_{\varepsilon_0}$,  and the measure disintegration of $\mu$ over $\eta$ is
\begin{equation*}
  \mu=\int\mu_{x}^{\eta}\textrm{d}\mu(x).
\end{equation*}
Then for any  $0<\delta<1$, we have
\begin{equation*}
\lim\limits_{\varepsilon \rightarrow 0}\liminf\limits_{n\rightarrow \infty}\frac{\log {N}^{u}_{\mu_{x}^{\eta}}(n, \varepsilon, \delta)}{n}=\lim\limits_{\varepsilon \rightarrow 0}\limsup\limits_{n\rightarrow \infty}\frac{\log {N}^{u}_{\mu_{x}^{\eta}}(n, \varepsilon, \delta)}{n}=h^{u}_{\mu}(f)
\end{equation*}
for $\mu$-a.e. $x\in M$, where ${N}^{u}_{\mu_{x}^{\eta}}(n, \varepsilon, \delta)$ denotes the minimal number of $ d_n^u$-balls with radius $\varepsilon$ whose union has  $\mu_{x}^{\eta}$-measure more than or equal to $1-\delta$.
\end{theorem}

See Section 2 for the definitions of $d_n^u$-ball, $\mathcal P^u_{\varepsilon_0}$ and $h^{u}_{\mu}(f)$.

Similarly, we have the following result:
\begin{theorem}\label{theorem3}
 Let $M$ be an $n$-dimensional smooth, connected and compact Riemannian manifold without boundary and  $f: M \rightarrow M$ be a $C^1$-partially hyperbolic diffeomorphism.	
Suppose $\mu$ is an ergodic measure of $f$. Let $\eta\in \mathcal P^u_{\varepsilon_0}$,
 and the measure disintegration of $\mu$ over $\eta$ is
\begin{equation*}
\mu=\int\mu_{x}^{\eta}\textrm{d}\mu(x).
\end{equation*}
Let  $\xi\in \mathcal P_{\varepsilon_0}$ satisfying $H_{\mu}(\xi|\eta)<\infty$.  Then for any  $0<\delta<1$, we have
\begin{equation*}
\lim\limits_{n\rightarrow \infty}\frac{\log {N}^{u}_{\mu_{x}^{\eta}}(n, \xi, \delta)}{n}=h^{u}_{\mu}(f)
\end{equation*}
for $\mu$-a.e. $x\in M$, where ${N}^{u}_{\mu_{x}^{\eta}}(n, \xi, \delta)$ denotes the minimal number of elements of the partition $\xi_0^{n-1}$ whose union has  $\mu_{x}^{\eta}$-measure more than or equal to $1-\delta$.
\end{theorem}
See  Section 2 for the definitions of $\xi_{0}^{n-1}$ and $\mathcal P_{\varepsilon_0}$.

The remainder of this paper is organized as follows. Section 2 gives some preliminaries. Section 3 provides the proof of Theorem \ref{theorem2}. Finally,  we prove Theorem \ref{theorem3} in Section 4.

\section{Preliminaries}
Let $M$ be an $n$-dimensional smooth, connected and compact Riemannian manifold without boundary and $f: M \rightarrow M$ be a $C^1$-diffeomorphism.
From now on we always assume that $f$ is a $C^1$-partially hyperbolic diffeomorphism of $M$, and $\mu$ is an $f$-invariant probability measure.

We say that $\alpha$ is a measurable partition of $M$ if there exists some measurable set $M_0\subset M$ with full measure such that, restricted to $M_0$,
\begin{equation*}
  \alpha=\bigvee_{n=1}^{\infty}\alpha_n
\end{equation*}
for some increasing sequence $\alpha_1\leq\alpha_2\leq\cdots\leq\alpha_n\leq\cdots$ of countable partitions. By $\alpha_i\leq \alpha_{i+1}$ we mean that every element of $\alpha_{i+1}$ is contained in some element of $\alpha_i$ or, equivalently, every element of $\alpha_i$ is a union of elements of $\alpha_{i+1}.$  Then we say that $\alpha_i$ is coarser than $\alpha_{i+1}$ or, equivalently, $\alpha_{i+1}$ is finer than $\alpha_i$.
Represent by  $\bigvee_{n=1}^{\infty}\alpha_n$ the partition whose elements are the non-empty intersections of the form $\cap_{n=1}^{\infty}A_n$ with $A_n\in \alpha_n$ for each $n$. Equivalently, this is the coarser partition such that
\begin{equation*}
  \alpha_i\leq \bigvee_{n=1}^{\infty}\alpha_n  \textrm{ for every } i.
\end{equation*}

It is easy to see that every countable partition is measurable.
 For a measurable partition $\xi$ of $M$ and $x\in M$, denote by $\xi(x)$ the element of $\xi$ containing $x$,
 and set
\begin{equation*}
  \xi_{0}^{n-1}=\xi \bigvee f^{-1}\xi\bigvee \cdots \bigvee f^{-(n-1)}\xi.
\end{equation*}

Let $\varepsilon_0$ be small enough and $\mathcal P_{\varepsilon_0}$ denote the set of finite measurable partitions of $M$ whose elements have diameters smaller than or equal to $\varepsilon_0$, that is, diam$\xi:=\sup\{\textrm{diam} A: A \in \xi\}\leq \varepsilon_0$. For each $\beta\in \mathcal P_{\varepsilon_0}$ we can define a finer partition $\eta$ such that $\eta(x)=\beta(x)\cap W^u_{\textrm{loc}}(x)$ for each $x\in M$, where  $W^u_{\textrm{loc}}(x)$ denotes the local unstable manifold at $x$  whose size is greater than the diameter $\varepsilon_0$ of $\beta$. Note that for each $x\in M$, $ W_{\textrm{loc}}^u(x)\subset W^u(x)$, where $W^u(x)$ is the unstable manifold at $x$, i.e.
\begin{equation*}
  W^u(x)=\{t\in M: \lim\limits_{n\rightarrow\infty}d(f^{-n}x,f^{-n}t)=0\},
\end{equation*}
where $d$ is the distance on  $M$ generated by the  Riemannian metric on $M$.

Note that if $\eta(x)=\beta(x)\cap W_{\textrm{loc}}^u(x)$, $\eta(y)=\beta(y)\cap W_{\textrm{loc}}^u(y)$ and $\eta(x)\cap\eta(y)\neq\emptyset$, then $\eta(x)=\eta(y)$. In fact, choose any $z\in \eta(x)\cap\eta(y) $.  Since $z\in \beta(x)\cap\beta(y) $, we obtain $\beta(x)=\beta(y)$. For any $t\in \eta(x)$, obviously, $t\in \beta(x)$, so $t\in\beta(y)$. And we also have $\lim\limits_{n\rightarrow\infty}d(f^{-n}t, f^{-n}x)=0$.  Moreover, observe that $\lim\limits_{n\rightarrow\infty}d(f^{-n}x, f^{-n}z)=\lim\limits_{n\rightarrow\infty}d(f^{-n}y, f^{-n}z)=0$ and
\begin{equation*}
  d(f^{-n}t, f^{-n}y)\leq d(f^{-n}t, f^{-n}x)+d(f^{-n}x, f^{-n}z)+d(f^{-n}z, f^{-n}y),
\end{equation*}
we have $\lim\limits_{n\rightarrow\infty}d(f^{-n}t, f^{-n}y)=0$, so $t\in W^u(y)$. Noting that  $t\in\beta(y)$ and the size of $W_{\textrm{loc}}^u(y)$ is greater than  $\textrm{diam}(\beta(y))$, we obtain $t\in W_{\textrm{loc}}^u(y)$. Therefore, $t\in \eta(y)=\beta(y)\cap W_{\textrm{loc}}^u(y)$. Then, we have $\eta(x)\subset\eta(y)$. Similarly, $\eta(y)\subset\eta(x)$. Hence, $\eta(x)=\eta(y)$.

Clearly, $\eta$ is a measurable partition satisfying $\beta\leq \eta$. Let $\mathcal P^u_{\varepsilon_0}$ denote the set of partitions $\eta$ obtained by this way.

Denote by $d^u$ the metric induced by the Riemannian structure on the unstable manifold. Given $\varepsilon>0$ and $x\in M$, let $B^u(x,\varepsilon)$ denote the open ball centered at $x$ with radius $\varepsilon$ in the unstable manifold $W^u(x)$ with respect to $d^u$. For  $x\in M$, $y\in W^u(x)$ and $n\in \mathbb N$, let $d^u_n(x,y)=\max\{d^u(f^{i}x, f^{i}y): i=0,1,\cdots, n-1\}$. Given $\varepsilon>0$, $x\in M$ and $n\in\mathbb N$, let $B^u_n(x, \varepsilon)=\{y\in W^u(x): d^u_n(x, y)<\varepsilon$\} denote the $d_n^u$-ball centered at $x$ with radius $\varepsilon$, i.e. an $(n,\varepsilon)$ Bowen ball in $W^u(x)$  centered at $x$.

The following conclusion is used in the proof of the Lemma 4.4 in \cite{[HU]}.
Let $\gamma>0$ be small enough. There exists $C>1$ such that for any $x\in M$,
\begin{equation*}
  d(y,z)\leq d^u(y,z)\leq Cd(y,z)
\end{equation*}
for any $y, z\in \overline{W^u(x, \gamma)}$, where $W^u(x,\gamma)$ is the open ball inside $W^u(x)$ centered at $x$ with radius $\gamma$ with respect to the metric $d^u$ and $\overline{W^u(x, \gamma)}$ is the closure of $W^u(x,\gamma)$.

Then, we can obtain the following two lemmas.

\begin{lemma}\label{lemma ball 1}
For $\gamma>0$ small enough,  $B^u_{n}(x, \gamma)\subset B_{n}(x, \gamma)$ for any $x\in M$.
\end{lemma}
\begin{proof}
Observe  that
\begin{equation*}
B^u_n(x,\gamma)=\cap_{i=0}^{n-1}f^{-i}B^u(f^{i}x, \gamma).
\end{equation*}
Noting that for $\gamma>0$ small enough, for each $x\in M$, $y, z\in \overline{W^u(x, \gamma)}$, we have $ d(y,z)\leq d^u(y,z)$. For any $y\in B^u_n(x,\gamma)$, $f^i y\in B^u(f^i x, \gamma)\subset \overline {W^u(f^i x, \gamma)}(i=0,1,\cdots, n-1)$, so
\begin{equation*}
  d(f^i x, f^i y)\leq d^u(f^i x, f^i y)<\gamma,
\end{equation*}
for $i=0, 1, \cdots, n-1$.
Thus $y\in B_n(x, \gamma)$.
Hence $B^u_{n}(x, \gamma)\subset B_{n}(x, \gamma)$.
\end{proof}

\begin{lemma}\label{lemma ball 2}
Given $\varepsilon_0>0$ small enough and $\eta\in \mathcal P_{\varepsilon_0}^{u}$, there exists $C>1$ such that  for each $x\in M$, $n\geq 1$, $0<\varepsilon<\varepsilon_0$ and $y\in \eta(x)$,
\begin{equation*}
  B_n(y, \frac{\varepsilon}{C})\cap\eta(x)\subset  B^u_n(y, \varepsilon)\cap\eta(x).
\end{equation*}
\end{lemma}
\begin{proof}

Since $f: M\rightarrow M$ is a $C^1$-diffeomorphism, for any $x\in M$, there exists $\lambda>1$ such that
\begin{equation*}
  d^u(fy, fz)<\lambda d^u(y,z),
\end{equation*}
for any $y, z\in W^u(x)$.
Note that for $\gamma>0$ small enough, there exists $C>1$ such that for any $x\in M$,
\begin{equation*}
  d(y,z)\leq d^u(y,z)\leq Cd(y,z)
\end{equation*}
for any $y, z\in \overline{W^u(x, \gamma)}$.
Therefore, for any $z\in  B_n(y, \frac{\varepsilon}{C})\cap\eta(x)$, we have
\begin{equation*}
  d^u(y, z)\leq C d(y,z)\leq\varepsilon_0.
\end{equation*}
Since $y,z\in W^u(x)$,  $d^u(y, z)\leq\varepsilon_0$ and $\varepsilon_0$ is small enough,
\begin{equation*}
  d^u(fy, fz)<\lambda d^u(y,z)<\lambda \varepsilon_0.
\end{equation*}
Since  $\varepsilon_0$ is small enough, $\lambda\varepsilon_0$ can be  smaller than  $\gamma$. And note that $f(y), f(z)\in W^u(fx)$, we have
\begin{equation*}
  d^u(fy,fz)\leq C d(fy,fz)\leq C\cdot\frac{\varepsilon}{C}=\varepsilon.
\end{equation*}
By  induction,
 we obtain
\begin{equation*}
  d^u(f^i y, f^i z)\leq Cd(f^i y,f^i z)<\varepsilon,
\end{equation*}
for $ i=0, 1, \cdots, n-1$.
Therefore $z\in B^u_n(y, \varepsilon)\cap\eta(x)$. Thus
\begin{equation*}
  B_n(y, \frac{\varepsilon}{C})\cap\eta(x)\subset  B^u_n(y, \varepsilon)\cap\eta(x).
\end{equation*}
\end{proof}

Recall that for a measurable partition $\eta$ of $M$ and a probability measure $\nu$ on $M$, the canonical system of conditional measures for $\nu$ and $\eta$ is a family of probability measures$\{\nu^{\eta}_x: x\in M\}$ with $\nu^{\eta}_x(\eta(x))=1$, such that for every measurable set $B\subset M$, $x\mapsto \nu^{\eta}_x(B)$ is measurable and
\begin{equation*}
  \nu(B)=\int_{M}\nu^{\eta}_x(B)\textrm{d}\nu(x).
\end{equation*}

The following notions are standard.

The information function of $\xi\in\mathcal P_{\varepsilon_0}$ with respect to $f$-invariant probability measure $\mu$ is defined as
\begin{equation*}
  I_{\mu}(\xi)(x):=-\log\mu(\xi(x)),
\end{equation*}
and the entropy of partition $\xi$ as
\begin{equation*}
  H_{\mu}(\xi):=\int_{M}I_{\mu}(\xi)(x)\textrm{d}\mu(x)=-\int_{M}\log\mu(\xi(x))\textrm{d}\mu(x).
\end{equation*}

The conditional information function of $\xi\in\mathcal P_{\varepsilon_{0}}$ given $\eta\in \mathcal P^u_{\varepsilon_0}$ with respect to $\mu$ is defined as
\begin{equation*}
  I_{\mu}(\xi|\eta)(x)=-\log\mu_{x}^{\eta}(\xi(x)).
\end{equation*}
Then the conditional entropy of $\xi\in\mathcal P_{\varepsilon_{0}}$ given $\eta\in \mathcal P^u_{\varepsilon_0}$  with respect to $\mu$ is defined as
\begin{equation*}
  H_{\mu}(\xi|\eta):=\int_{M}I_{\mu}(\xi|\eta)(x)\textrm{d}\mu(x)=-\int_M\log\mu_x^{\eta}(\xi(x))\textrm{d}\mu(x).
\end{equation*}
\begin{definition}\cite{[HU]}
For an $f$-invariant probability measure $\mu$, the conditional entropy of $f$ with respect to $\xi\in\mathcal P_{\varepsilon_0}$ given $\eta\in \mathcal P_{\varepsilon_{0}}^u$ is defined as
\begin{equation*}
 h_{\mu}(f, \xi|\eta)=\limsup\limits_{n\rightarrow \infty}\frac{1}{n}H_{\mu}(\xi^{n-1}_{0}|\eta).
\end{equation*}
The conditional entropy of $f$ given $\eta\in \mathcal P_{\varepsilon_0}^u$ with respect to $\mu$ is defined as
\begin{equation*}
  h_{\mu}(f|\eta)=\sup_{\xi\in\mathcal P_{\varepsilon_0}}h_{\mu}(f,\xi|\eta).
\end{equation*}
And the unstable metric entropy of $f$ with respect to $\mu$ is defined as
\begin{equation*}
  h_{\mu}^u(f)=\sup_{\eta\in\mathcal P_{\varepsilon_0}^u}h_{\mu}(f|\eta).
\end{equation*}
\end{definition}

The following theorems will be used in  proving the main results.
\begin{theorem}\cite{[HU]}\label{Theorem 2.1}
Suppose $\mu$ is an $f$-invariant probability measure. Then for any $\xi\in\mathcal P_{\varepsilon_0}$ and $\eta\in \mathcal P^u_{\varepsilon_0}$,
\begin{equation*}
  h_{\mu}^u(f)=h_{\mu}(f|\eta)=h_{\mu}(f, \xi|\eta).
\end{equation*}
\end{theorem}
\begin{theorem}\label{theorem 2.2}\cite{[HU]}
Suppose $\mu$ is an ergodic measure of $f$. Let $\eta \in \mathcal P^u_{\varepsilon_0}$ be given. Then for any partition $\xi\in \mathcal P_{\varepsilon_0}$ with $H_{\mu}(\xi|\eta)<\infty$, we have
\begin{equation*}
 \lim\limits_{n\rightarrow \infty}\frac{I_{\mu}(\xi_{0}^{n-1}|\eta)(x)}{n}= h_{\mu}(f,\xi|\eta),
\end{equation*}
for $\mu$-a.e. $x\in M$.
\end{theorem}

The result in Theorem \ref{theorem 2.2} is a version of Shannon-McMillan-Breiman theorem for the unstable metric entropy.


\section{Proof of Theorem \ref{theorem2}}
This section gives the proof of Theorem \ref{theorem2}. \begin{lemma}\label{lemma3.2}
Suppose $\mu$ is an ergodic measure of $f$. Let $\eta \in \mathcal P^u_{\varepsilon_0}$ be given, and $\xi\in \mathcal P_{\varepsilon_0}$ with $H_{\mu}(\xi|\eta)<\infty$.
 Then  there exists a set $M_1\subset M$ with $\mu(M_1)=1$, and for each $x\in M_1$, there exists $G_x\subset \eta(x)$ with $\mu_{x}^{\eta}(G_x)=1$ such that
\begin{equation*}
  \lim\limits_{n\rightarrow \infty}\frac{-\log\mu_{x}^{\eta}(\xi_{0}^{n-1}(y))}{n}= h_{\mu}(f, \xi|\eta)
\end{equation*}
for each $y\in G_x$, where $\mu=\int \mu_{x}^{\eta}\textrm{d}\mu(x)$ is the measure disintegration of $\mu$ over $\eta$.
\end{lemma}
\begin{proof}
By Theorem \ref{theorem 2.2}, there exists a set $M_1\subset M$ with $\mu(M_1)=1$, such that
\begin{equation*}
  \lim\limits_{n\rightarrow \infty}\frac{-\log\mu_{x}^{\eta}(\xi_{0}^{n-1}(x))}{n}= h_{\mu}(f, \xi|\eta)
\end{equation*}
for each $x\in M_1$. Then for each $x\in M_1$, we can find a set $G_x\subset \eta(x)$ with $\mu^{\eta}_x(G_x)=1$ such that
\begin{equation*}
  \lim\limits_{n\rightarrow \infty}\frac{-\log\mu_{y}^{\eta}(\xi_{0}^{n-1}(y))}{n}= h_{\mu}(f, \xi|\eta)
\end{equation*}
for each $y\in G_x$. Note that for each $y\in G_x$, $\mu_x^{\eta}=\mu_y^{\eta}$. Therefore, for each $x\in M_1$ and $y\in G_x$, we have
\begin{equation*}
  \lim\limits_{n\rightarrow \infty}\frac{-\log\mu_{x}^{\eta}(\xi_{0}^{n-1}(y))}{n}= h_{\mu}(f, \xi|\eta).
\end{equation*}
\end{proof}

Now, we are going to prove Theorem \ref{theorem2}.

\begin{proof}
(1) Firstly, we are going to show that for every $0<\delta<1$, we have
\begin{equation*}
\lim\limits_{\varepsilon \rightarrow 0}\limsup\limits_{n\rightarrow \infty}\frac{\log N^{u}_{\mu_{x}^{\eta}}(n, \varepsilon, \delta)}{n}\leq h^u_{\mu}(f)
\end{equation*}
for   $\mu$-a.e. $x\in M$.

For $0<\varepsilon<\varepsilon_0$, let us choose a  finite  partition $\xi\in \mathcal P_{\varepsilon_0}$  with
 diam$(\xi)<\varepsilon/C$ and $H_{\mu}(\xi|\eta)<\infty$.
Since $\varepsilon_0$ is small enough, by Lemma \ref{lemma ball 2},  for each $x\in M$, $n\geq 1$ and $y\in \eta(x)$, we have
 \begin{equation*}
   \xi_0^{n-1}(y)\cap \eta(x)\subset B_n(y, \frac{ \varepsilon}{C})\cap  \eta(x)\subset B^u_{n}(y, \varepsilon)\cap  \eta(x).
 \end{equation*}
   Observe that $\mu$ is ergodic. According to Lemma \ref{lemma3.2}, there exists a subset $M_1\subset M$ with $\mu(M_1)=1$ such that for any $x\in M_1$, there exists a set $ G_{x}\subset\eta(x)$ with $\mu_{x}^{\eta}( G_{x})=1$ such that for any $y\in G_{x}$,
 \begin{equation*}
   \lim\limits_{n\rightarrow \infty}\frac{-\log\mu_{x}^{\eta}(\xi_0^{n-1}(y))}{n}=h_{\mu}(f, \xi|\eta).
 \end{equation*}
 Fix $x\in M_1$. For $n\in \mathbb N$ and $\gamma>0$, set
 \begin{equation*}
   Y_{n} = \{y\in  G_{x}: \mu_{x}^{\eta}(\xi_{0}^{n-1}(y))>\exp(-(h_{\mu}(f, \xi|\eta)+\gamma)n)\}\subset  \bigcup_{V\in\mathcal J_n}V\cap \eta(x),
 \end{equation*}
where $\mathcal J_n=\{V\in \xi_0^{n-1}: \mu_{x}^{\eta}(V)>\exp(-(h_{\mu}(f, \xi|\eta)+\gamma)n)\}.$
Then for each $\gamma >0$, $ \lim\limits_{n\rightarrow \infty}\mu_{x}^{\eta}(Y_{n})=1.$
Thus, for sufficiently large $n\in \mathbb N$, we have $\mu_{x}^{\eta}(Y_{n})>1-\delta$.  Since
\begin{equation*}
\#\mathcal J_n=\# \{V\in \xi_0^{n-1}: \mu_{x}^{\eta}(V)>\exp(-(h_{\mu}(f, \xi|\eta)+\gamma)n)\}\leq \exp((h_{\mu}(f, \xi|\eta)+\gamma)n),
\end{equation*}
the set $Y_{n}$ contains at most $\exp((h_{\mu}(f, \xi|\eta)+\gamma)n)$ elements  of  $\xi_{0}^{n-1}\cap \eta(x)$, where $\xi_{0}^{n-1}\cap \eta(x)=\{V\cap \eta(x): V\in\xi_{0}^{n-1}\}$. Noting that for any $y \in\eta(x)$, $ \xi_0^{n-1}(y)\cap  \eta(x)\subset B^u_{n}(y, \varepsilon)\cap  \eta(x)$, so $Y_n$ can be covered by $ d_n^u$-balls with radius $\varepsilon$ of  the same number. So
\begin{equation*}
 {N}^u_{\mu_{x}^{\eta}}(n, \varepsilon, \delta)\leq \exp((h_{\mu}(f, \xi|\eta)+\gamma)n).
\end{equation*}
Then for any $\gamma >0$,
\begin{equation*}
  \lim\limits_{\varepsilon \rightarrow 0}\limsup\limits_{n\rightarrow \infty}\frac{\log {N}^u_{\mu_{x}^{\eta}}(n, \varepsilon, \delta)}{n}\leq h_{\mu}(f, \xi|\eta)+\gamma.
\end{equation*}
Since $\gamma$ can be taken arbitrarily small and by Theorem \ref{Theorem 2.1} $h_{\mu}(f, \xi|\eta)=h^u_{\mu}(f)$, we obtain
\begin{equation*}
  \lim\limits_{\varepsilon \rightarrow 0}\limsup\limits_{n\rightarrow \infty}\frac{\log {N}^u_{\mu_{x}^{\eta}}(n, \varepsilon, \delta)}{n}\leq h^u_{\mu}(f)
\end{equation*}
for every $x\in M_1$. Noting that $\mu(M_1)=1$, we have
\begin{equation*}
  \lim\limits_{\varepsilon \rightarrow 0}\limsup\limits_{n\rightarrow \infty}\frac{\log {N}^u_{\mu_{x}^{\eta}}(n, \varepsilon, \delta)}{n}\leq h^u_{\mu}(f)
\end{equation*}
 for $\mu$-a.e. $x\in M$.

(2) Secondly, we will turn to prove the second part of the theorem: for every $0<\delta<1$, we have
\begin{equation*}
\lim\limits_{\varepsilon \rightarrow 0}\liminf\limits_{n\rightarrow \infty}\frac{\log {N}^u_{\mu_{x}^{\eta}}(n, \varepsilon, \delta)}{n}\geq h^u_{\mu}(f)
\end{equation*}
for $\mu$-a.e. $x\in M$.

(i) Let $0<\delta<1$ be given. Let $\varepsilon>0$, without loss of generality, we require additionally $\varepsilon^{\frac{1}{4}}<\frac{1-\delta}{4}$. Let us choose a partition $\xi \in \mathcal P_{\varepsilon_0}$ with $\mu(\partial \xi)=0$ and $H_{\mu}(\xi|\eta)<\infty$, where $\partial \xi$ denotes  union of the boundaries $\partial B$ of all elements $B \in \xi$.

For $ \theta >0$, let
\begin{equation*}
U_{\theta}(\xi)=\{x \in M: \textrm{ the ball } B(x, \theta) \textrm{ is not containded in } \xi(x) \},
\end{equation*}
where $\xi(x)$ denotes the element of the partition $\xi$ containing $x$. Since $\bigcap_{\theta >0}U_{\theta}(\xi)=\partial{\xi}$, we have
\begin{equation*}
\mu\left(U_{\theta}(\xi)\right)\rightarrow 0 , \textrm{ as }  \theta \rightarrow 0.
\end{equation*}
 Therefore, there exists $0<\gamma<\varepsilon$  such that $\mu(U_{\theta}(\xi))\leq \varepsilon$ for any $0 < \theta \leq \gamma$.
Using  Birkhoff Ergodic Theorem,
 for $\mu$-a.e. $y\in M$ there exists $N_1(y)>0$ such that for any $k\geq N_1(y)$,
 \begin{equation*}
   \frac{1}{k}\sum_{i=0}^{k-1}\chi_{U_{\gamma}(\xi)}(f^i(y))\leq { \varepsilon},
 \end{equation*}
 where $\chi_{U_{\gamma}(\xi)}$ is the characteristic function of the set $U_{\gamma}(\xi)$.
  For $l\in \mathbb N^+$, we define
  \begin{equation*}
    D_l=\left\{y\in M: \frac{1}{k}\sum_{i=0}^{k-1}\chi_{U_{\gamma}(\xi)}(f^i(y))\leq {\varepsilon} \textrm{ for any } k\geq l\right \}.
  \end{equation*}
Clearly, the sets $D_l$ are nested and exhaust $M$ up to a set of $\mu$-measure zero. Therefore, there exists $l_0>1$ such that $\mu(D_l)\geq 1-2\sqrt{\varepsilon}$ for any $l\geq l_0$.

 Define $M_l=\{x\in M: \mu_{x}^{\eta}(D_l)\geq 1-2\varepsilon^{\frac{1}{4}}\}$, then $M_l^C=\{x\in M: \mu_{x}^{\eta}(D_l^C)\geq 2\varepsilon^{\frac{1}{4}}\}$, where for any set $A\subset M$, $A^C$ is the complement of  $A$. Using Chebyshev's inequality,  we obtain
 \begin{equation*}
   \mu(M_l^C)=\int\mu_{x}^{\eta}(M_l^C)\textrm{d}\mu(x)\leq\frac{\int \mu_{x}^{\eta}(D_l^C)\textrm{d}\mu(x)}{2\varepsilon^{\frac{1}{4}}}=\frac{\mu(D_l^C)}{2\varepsilon^{\frac{1}{4}}}\leq\frac{2\sqrt{\varepsilon}}
   {2\varepsilon^{\frac{1}{4}}}=\varepsilon^{\frac{1}{4}},
 \end{equation*}
 for any $l\geq l_0$.
Thus for any $l\geq l_0$, $\mu(M_l)\geq 1-\varepsilon^{\frac{1}{4}}$. The sets $D_l$ are nested, i.e. $D_1\subset D_2\subset\cdots$. Then fix some $l_1>l_0$, for any $x\in M_{l_{1}}$, $l\geq l_1$ we have
\begin{equation}\label{9}
  \mu_{x}^{\eta}(D_l)\geq\mu_{x}^{\eta}(D_{l_{1}})\geq 1-2\varepsilon^{\frac{1}{4}}.
\end{equation}
 According to Lemma \ref{lemma3.2}, we can find a subset $M_1\subset M$ with $\mu(M_1)=1$ such that for any $x\in M_1$, there exists set $G_x$ with $\mu_{x}^{\eta}(G_x)=1$ such that for any $y\in G_x$,
 \begin{equation*}
   \lim\limits_{n\rightarrow \infty}\frac{-\log\mu_{x}^{\eta}(\xi_0^{n-1}(y))}{n}=h_{\mu}(f, \xi|\eta).
 \end{equation*}
Let $I=M_1\cap M_{l_{1}}$. Clearly, $\mu(I)\geq 1-\varepsilon^{\frac{1}{4}}$.

(ii) For $n\in \mathbb N$ and given a point $y\in M$, we call the collection
\begin{equation*}
C(n, y):=({\xi}(y), {\xi}(f(y)), \cdots, {\xi}(f^{n-1}(y))
\end{equation*}
the $(\xi, n)$-name of $y$.  Since each point in one element $V$ of $\xi_{0}^{n-1}$ has the same $(\xi, n)$-name, we can define
\begin{equation*}
  C(n, V):=C(n, y)
\end{equation*}
for any $y\in V$, which is called the $(\xi, n)$-name of $V$.

For $n \in \mathbb{N}$ and $\xi$, we give a metric $d_{n}^{\xi}$ between  $(\xi, n)$-names of $y$ and $z$ as follows:
\begin{equation*}
d_{n}^{\xi}(C(n, y), C(n, z))=\frac{1}{n}\#\{0 \leq i \leq n-1 : {\xi}(f^i(y))\neq {\xi}(f^i(z))\}.
\end{equation*}
It can also be viewed as a semi-metric on $M$.

Fix $\hat x\in I$ and $l_2\geq l_1$. According to Lemma \ref{lemma ball 1}, for $\gamma>0$ small enough,  $B^u_{n}(y, {\gamma})\subset B_n(y, \gamma)$, for any $y\in M$. If $z\in B(y,\gamma)$, then either $y$ and
$z$  belong to the same element of $\xi$ or $y\in U_{\gamma}(\xi)$, $z\notin \xi(y)$. Hence if $y \in D_{l_2}$, $n\geq l_2$ and $z \in B^u_{n}(y, {\gamma})$, the distance $d_{n}^{\xi}$ between $(\xi, n)$-names of $y$ and $z$ does not exceed ${\varepsilon}$, i.e.
\begin{equation*}
  d_{n}^{\xi}(C(n, y), C(n, z))\leq {\varepsilon}.
\end{equation*}
 Furthermore, for $y \in D_{l_2}$, $n\geq l_2$, $B^u_{n}(y, \gamma)$ is contained in the set of points $z$ whose $(\xi, n)$-names are ${\varepsilon}$-close to the $(\xi, n)$-name of $y$, i.e.
\begin{equation}\label{14}
B^u_{n}(y, \gamma)\subset B_{d_{n}^{\xi}}(y, {\varepsilon}).
\end{equation}

By Stirling's formula, there exists a large number $l_3\in \mathbb N$ and for any $n\geq l_3$, it can be shown that the total number $K_n$ of such $(\xi, n )$-names consisting of $B_{d_{n}^{\xi}}(y, {\varepsilon})$ admits the following estimate:
\begin{equation}\label{15}
K_n  \leq \sum_{j=0}^{[n{\varepsilon}]}C_{n}^{j}(\#\xi-1)^j \leq \sum_{j=0}^{[n{\varepsilon}]}C_{n}^{j}(\#\xi)^j\leq \exp(({\varepsilon} + \Diamond)n),
\end{equation}
where
\begin{equation*}
  \Diamond={\varepsilon}\log(\#\xi)-{\varepsilon}\log{\varepsilon}
  -(1-{\varepsilon})\log(1-{\varepsilon}).
\end{equation*}

For $n\geq\max\{l_2, l_3\}$, set
\begin{equation*}
  \mathcal U:=\left\{B^u_{n}(y_i, \frac{\gamma}{2}): i=1, 2, \cdots, N^u_{\mu_{\hat x}^{\eta}}(n, \frac{\gamma}{2}, \delta)\right\}
\end{equation*}
with $\mu_{\hat x}^{\eta}(F_n)>1-\delta$, where
\begin{equation*}
  F_n:=\bigcup_{i=1}^{ N^u_{\mu_{\hat x}^{\eta}}(n, \frac{\gamma}{2}, \delta)}B^u_{n}(y_i, \frac{\gamma}{2}).
\end{equation*}
According to  (\ref{9}) and  $\varepsilon^{\frac{1}{4}}<\frac{1-\delta}{4}$, then $\mu_{\hat x}^{\eta}(F_n\cap D_n)>1-\delta-2\varepsilon^{\frac{1}{4}}>\frac{1-\delta}{2}$. For $i=1, 2, \cdots,  N^u_{\mu_{\hat x}^{\eta}}(n, \frac{\gamma}{2}, \delta)$, if $B^u_{n}(y_i, \frac{\gamma}{2})\cap D_n\neq \emptyset$, we choose any $z_i\in B^u_{n}(y_i, \frac{\gamma}{2})\cap D_n$. Then  we apply the relation (\ref{14}), so we have
\begin{equation*}
   B^u_{n}(y_i, \frac{\gamma}{2})\cap D_n\subset B^u_{n}(z_i, \gamma)\subset B_{d_n^{\xi}}(z_i, {\varepsilon}).
\end{equation*}
Thus,
\begin{equation*}
  F_n\cap D_n\subset S_n,
\end{equation*}
where
\begin{equation*}
 S_n=\bigcup_{\{i:B^u_{n}(y_i, \frac{\gamma}{2})\cap D_n\neq \emptyset\}} B_{d_n^{\xi}}(z_i, {\varepsilon}).
\end{equation*}
Let
\begin{equation*}
  \mathcal P_n=\left\{V\in \xi_{0}^{n-1} :  d_n^{\xi}(C(n, V), C(n, z_i))<{\varepsilon}, \textrm{ for some } i=1, 2, \cdots, N^u_{\mu_{\hat x}^{\eta}}(n, \frac{\gamma}{2}, \delta)\right\}.
\end{equation*}
It is clear that $S_n=\bigcup_{V\in \mathcal P_n}V$ and
\begin{equation*}
  \#\mathcal P_n\leq  N^u_{\mu_{\hat x}^{\eta}}(n, \frac{\gamma}{2}, \delta)\cdot K_n.
\end{equation*}

By Lemma \ref{lemma3.2} and Egorov Theorem, there exists a large number $l_4>\max\{l_2, l_3\}$ such that, $\mu_{\hat x}^{\eta}(T_n)\geq
\frac{1-\delta}{4}$ for each $n\geq l_4$, where
\begin{equation*}
  T_n=\{y\in S_n: \mu_{\hat x}^{\eta}(\xi_0^{n-1}(y))\leq \exp(-(h_{\mu}(f, \xi|\eta)-\varepsilon)n)\}.
\end{equation*}
Write $t_n:=\#\{\xi_0^{n-1}(y): y\in T_n\}$.  Then
\begin{equation*}
  \frac{(1-\delta)\exp((h_{\mu}(f, \xi|\eta)-\varepsilon)n)}{4}\leq t_n\leq \#\mathcal P_n\leq N^u_{\mu_{\hat x}^{\eta}}(n, \frac{\gamma}{2}, \delta)\cdot K_n.
\end{equation*}
Hence, we have
\begin{equation*}
  N^u_{\mu_{\hat x}^{\eta}}(n, \frac{\gamma}{2}, \delta)\geq  \frac{(1-\delta)\exp((h_{\mu}(f, \xi|\eta)-\varepsilon)n)}{4K_n}.
\end{equation*}
Noting that by Theorem \ref{Theorem 2.1} $h_{\mu}(f, \xi|\eta)=h^u_{\mu}(f)$, and using (\ref{15}), we obtain
\begin{eqnarray*}
  \liminf\limits_{n\rightarrow \infty}\frac{\log   N^u_{\mu_{\hat x}^{\eta}}(n, \frac{\gamma}{2}, \delta)}{n} &\geq& h_{\mu}(f, \xi|\eta)-\varepsilon-\limsup\limits_{n\rightarrow \infty}\frac{\log K_n}{n}+\lim\limits_{n\rightarrow \infty}\frac{1}{n}\log\frac{1-\delta}{4} \\
  &\geq& h_{\mu}(f, \xi|\eta)-\varepsilon-({\varepsilon}+\Diamond)\\
  &=& h^u_{\mu}(f)-\varepsilon-({\varepsilon}+\Diamond)\\
  &=&  h^u_{\mu}(f)-2\varepsilon-\Diamond.
\end{eqnarray*}
Let $\varepsilon\rightarrow 0$. Since $\gamma<\varepsilon$, $\lim\limits_{\varepsilon\rightarrow 0}\Diamond=0$, $\hat x\in I$ and $\mu(I)\geq 1-\varepsilon^{\frac{1}{4}}$, we have
\begin{equation*}
  \lim\limits_{\gamma\rightarrow 0} \liminf\limits_{n\rightarrow \infty}\frac{ \log N^u_{\mu_{\hat x}^{\eta}}(n, \frac{\gamma}{2}, \delta)}{n}\geq h^u_{\mu}(f),
\end{equation*}
for $\mu$-a.e. $\hat x\in M$.

Therefore, for every $0<\delta<1$, we obtain
\begin{equation*}
\lim\limits_{\varepsilon \rightarrow 0}\liminf\limits_{n\rightarrow \infty}\frac{\log {N}^u_{\mu_{x}^{\eta}}(n, \varepsilon, \delta)}{n}\geq h^u_{\mu}(f)
\end{equation*}
for $\mu$-a.e. $x\in M$.
\end{proof}

\section{Proof of Theorem \ref{theorem3}}
In this section, we will prove Theorem \ref{theorem3}.
\begin{proof}
(1) Firstly, we will prove that for each $0<\delta<1$,
\begin{equation*}
\limsup\limits_{n\rightarrow \infty}\frac{\log {N}^{u}_{\mu_{x}^{\eta}}(n, \xi, \delta)}{n}\leq h^{u}_{\mu}(f)
\end{equation*}
for   $\mu$-a.e. $x\in M$.

   Noting that $\mu$ is ergodic and according to Lemma \ref{lemma3.2}, there exists a subset $M_1\subset M$ with $\mu(M_1)=1$ such that for any $x\in M_1$, there exists a set $G_{x}\subset\eta(x)$ with $\mu_{x}^{\eta}(G_{x})=1$ such that for any $y\in G_{x}$,
 \begin{equation*}
   \lim\limits_{n\rightarrow \infty}\frac{-\log\mu_{x}^{\eta}(\xi_0^{n-1}(y))}{n}=h_{\mu}(f, \xi|\eta).
 \end{equation*}
 Fix $x\in M_1$. For $n\in \mathbb N$ and $\gamma>0$, set
 \begin{equation*}
   Y_{n} = \{y\in  G_{x}: \mu_{x}^{\eta}(\xi_{0}^{n-1}(y))>\exp(-(h_{\mu}(f, \xi|\eta)+\gamma)n)\}\subset  \bigcup_{V\in\mathcal J_n}V\cap \eta(x),
 \end{equation*}
where $\mathcal J_n=\{V\in \xi_0^{n-1}: \mu_{x}^{\eta}(V)>\exp(-(h_{\mu}(f, \xi|\eta)+\gamma)n)\}.$
Then for each $\gamma >0$, $ \lim\limits_{n\rightarrow \infty}\mu_{x}^{\eta}(Y_{n})=1.$
Thus, for sufficiently large $n\in \mathbb N$, we have $\mu_{x}^{\eta}(Y_{n})>1-\delta$.  Since
\begin{equation*}
\#\mathcal J_n=\# \{V\in \xi_0^{n-1}: \mu_{x}^{\eta}(V)>\exp(-(h_{\mu}(T, \xi|\eta)+\gamma)n)\}\leq \exp((h_{\mu}(f, \xi|\eta)+\gamma)n),
\end{equation*}
the set $Y_{n}$ contains at most $\exp((h_{\mu}(f, \xi|\eta)+\gamma)n)$ elements  of  $\xi_{0}^{n-1}\cap \eta(x)$, where $\xi_{0}^{n-1}\cap \eta(x)=\{V\cap \eta(x): V\in\xi_{0}^{n-1}\}$. Thus
\begin{equation*}
 {N}^u_{\mu_{x}^{\eta}}(n, \xi, \delta)\leq \exp((h_{\mu}(f, \xi|\eta)+\gamma)n).
\end{equation*}
Then for any $\gamma >0$,
\begin{equation*}
\limsup\limits_{n\rightarrow \infty}\frac{\log {N}^u_{\mu_{x}^{\eta}}(n, \xi, \delta)}{n}\leq h_{\mu}(f, \xi|\eta)+\gamma.
\end{equation*}
Since $\gamma$ can be taken arbitrarily small and by Theorem \ref{Theorem 2.1} $h_{\mu}(f, \xi|\eta)=h^u_{\mu}(f)$, we obtain
\begin{equation*}
  \limsup\limits_{n\rightarrow \infty}\frac{\log {N}^u_{\mu_{x}^{\eta}}(n, \xi, \delta)}{n}\leq h^u_{\mu}(f)
\end{equation*}
for every $x\in M_1$. Noting that $\mu(M_1)=1$, we have
\begin{equation*}
 \limsup\limits_{n\rightarrow \infty}\frac{\log {N}^u_{\mu_{x}^{\eta}}(n, \xi, \delta)}{n}\leq h^u_{\mu}(f)
\end{equation*}
 for $\mu$-a.e. $x\in M$.

(2) Secondly, we will turn to prove that for every $0<\delta<1$,
\begin{equation*}
\liminf\limits_{n\rightarrow \infty}\frac{\log {N}^u_{\mu_{x}^{\eta}}(n, \xi, \delta)}{n}\geq h^u_{\mu}(f)
\end{equation*}
for $\mu$-a.e. $x\in M$.

 Given $\eta\in \mathcal P_{\varepsilon_0}^u$ and  $\xi \in \mathcal P_{\varepsilon_0}$ with   $H_{\mu}(\xi|\eta)<\infty$.
By Lemma \ref{lemma3.2}, we can find a subset $M_1\subset M$ with $\mu(M_1)=1$ such that for any $x\in M_1$, there exists set $G_x$ with $\mu_{x}^{\eta}(G_x)=1$ such that for any $y\in G_x$,
 \begin{equation*}
   \lim\limits_{n\rightarrow \infty}\frac{-\log\mu_{x}^{\eta}(\xi_0^{n-1}(y))}{n}=h_{\mu}(f, \xi|\eta).
 \end{equation*}
Now, fix $\hat x\in M_1$.
For $n\in \mathbb N$, set
\begin{equation*}
  \mathcal U:=\left\{\xi_{0}^{n-1}(y_i): i=1, 2, \cdots, N^u_{\mu_{\hat x}^{\eta}}(n, \xi, \delta)\right\}
\end{equation*}
with $\mu_{\hat x}^{\eta}(F_n)>1-\delta$, where
\begin{equation*}
  F_n:=\bigcup_{i=1}^{ N^u_{\mu_{\hat x}^{\eta}}(n, \xi, \delta)}\xi_{0}^{n-1}(y_i).
\end{equation*}

 Let $\varepsilon>0$ small enough.
By Lemma \ref{lemma3.2} and Egorov Theorem, there exists a large number $N\in \mathbb N$ such that, $\mu_{\hat x}^{\eta}(T_n)\geq
(1-\delta)/2$ for each $n\geq N$, where
\begin{equation*}
  T_n=\{y\in F_n: \mu_{\hat x}^{\eta}(\xi_0^{n-1}(y))\leq \exp(-(h_{\mu}(f, \xi|\eta)-\varepsilon)n)\}.
\end{equation*}
Write $t_n:=\#\{\xi_0^{n-1}(y): y\in T_n\}$.  Then
\begin{equation*}
  \frac{(1-\delta)\exp((h_{\mu}(f, \xi|\eta)-\varepsilon)n)}{2}\leq t_n\leq  N^u_{\mu_{\hat x}^{\eta}}(n, \xi, \delta).
\end{equation*}

Noting that by Theorem \ref{Theorem 2.1} $h_{\mu}(f, \xi|\eta)=h^u_{\mu}(f)$, we obtain
\begin{eqnarray*}
  \liminf\limits_{n\rightarrow \infty}\frac{\log   N^u_{\mu_{\hat x}^{\eta}}(n, \xi, \delta)}{n} &\geq& h_{\mu}(f, \xi|\eta)-\varepsilon+\lim\limits_{n\rightarrow \infty}\frac{1}{n}\log\frac{1-\delta}{2} \\
  &=& h_{\mu}(f, \xi|\eta)-\varepsilon\\
  &=& h^u_{\mu}(f)-\varepsilon.
\end{eqnarray*}
Let $\varepsilon\rightarrow 0$,  we have
\begin{equation*}
\liminf\limits_{n\rightarrow \infty}\frac{ \log N^u_{\mu_{\hat x}^{\eta}}(n, \xi, \delta)}{n}\geq h^u_{\mu}(f).
\end{equation*}
Note that $\hat x\in M_1$ and $\mu(M_1)=1$.
Therefore, for every $0<\delta<1$, we obtain
\begin{equation*}
\liminf\limits_{n\rightarrow \infty}\frac{\log {N}^u_{\mu_{x}^{\eta}}(n, \xi, \delta)}{n}\geq h^u_{\mu}(f)
\end{equation*}
for $\mu$-a.e. $x\in M$.

\end{proof}

\textbf{Acknowlegments:}

The first  author was  supported by NNSF of China (11401581, 11971236). The second author was supported by NNSF of China (11671208, 11431012). And the third author was  supported by NNSF of China (11401581). At the end, we would like to express our gratitude to Tianyuan Mathematical Center in Southwest China, Sichuan University and Southwest Jiaotong University for their support and hospitality.

 \small

\end{document}